\documentclass[a4paper,12pt, oneside]{amsart}

\usepackage{amsthm, amsmath}
\usepackage{amssymb,amsfonts,amscd,verbatim, longtable}
\usepackage{tikz}
\usepackage{ifpdf}
\usepackage{float}

\usepackage{xy}
\usepackage{mathrsfs}
\usepackage{wrapfig}
\xyoption{all}
\SelectTips{cm}{10}

\ifpdf
\usepackage{hyperref}
\else
\usepackage[hypertex]{hyperref}
\fi
\newtheorem{tr}{Theorem}[section]
\newtheorem*{tr*}{Theorem}
\newtheorem{lemma}[tr]{Lemma}
\newtheorem{pr}[tr]{Proposition}

\newtheorem{cor}[tr]{Corollary}

\theoremstyle{remark}
\newtheorem{rem}[tr]{Remark}

\newlength{\myevenmargin}
\def\differential{d}
\renewcommand\d\differential

\DeclareMathOperator\im{Im}

\DeclareMathOperator\CHAR{char}

\DeclareMathOperator\Aut{Aut}

\DeclareMathOperator\Ext{Ext}

\DeclareMathOperator\Span{Span}

\DeclareMathOperator\Bl{Bl}

\DeclareMathOperator\Tsp{T}

\def\k{\Bbbk}
\renewcommand\Im\im
\def\bb#1{\mathbb #1}
\def\cal#1{\mathcal #1}

\def\ra{\rightarrow}
\def\xra{\xrightarrow}

\def\mto{\mapsto}

\def\pmat#1{\begin{pmatrix}#1\end{pmatrix}}

\def\smat#1{\left(\begin{smallmatrix}#1\end{smallmatrix}\right)}

\def\point#1{\langle #1 \rangle}

\def\refeq#1{$(\ref{#1})$}

\def\P{\bb P}

\def\O{\cal O}

\let\star *
\let\subset\subseteq
\let\supset\supseteq

\def\defeq{:=}

\def\iso{\cong}

\def\tilde{\widetilde}

\title[On moduli of $1$-dimensional sheaves on plane quartics]{A global description of the fine Simpson moduli spaces of $1$-dimensional sheaves supported on plane quartics}
\author{Oleksandr Iena}
\address{University of Luxembourg, Campus Kirchberg\\
Mathematics Research Unit\\
6, rue Richard Coudenhove-Kalergi\\
L-1359 Luxembourg City\\
Grand Duchy of Luxembourg}
\email{oleksandr.iena@uni.lu}
\date{}

\begin{document}

\begin{abstract}
 A global description of the fine Simpson moduli spaces of $1$-dimensional sheaves supported on plane quartics is given: we describe the gluing of the Brill-Noether loci described by Dr\'ezet and Maican and show that the Simpson moduli space $M=M_{4m\pm 1}(\P_2)$ is a blow-down of a blow-up of a projective bundle over a smooth moduli space of Kronecker modules. An easy computation of the Poincar\'e polynomial of $M$ is presented.
\end{abstract}
\maketitle

\section*{Introduction}
 Fix an algebraically closed field $\k$, $\CHAR \k=0$. Let $V$ be a $3$-dimensional vector space over $\k$ and let $\P_2=\P V$ be the corresponding projective plane. Let $P(m)=dm+c$, $d\in \bb Z_{>0}$, $c\in \bb Z$ be a linear Hilbert polynomial. Let $M=M_{dm+c}=M_{dm+c}(\P_2)$ be the Simpson moduli  space (cf.~\cite{Simpson1}) of semi-stable sheaves on $\P_2$ with Hilbert polynomial $dm+c$.

In~\cite{MaicanDuality} and~\cite{Woolf} it has been shown that
$M_{dm+c}\iso M_{d'm+c'}$ iff $d=d'$ and $c=\pm c' \mod d$.
Therefore, in order to understand,  for fixed $d$, the Simpson moduli spaces $M_{dm+c}$ it is enough to understand at most $d/2+1$ different moduli spaces.

If $\gcd(d, c)=1$, every semi-stable sheaf is stable and $M_{dm+c}$ is a fine moduli space whose points correspond to the isomorphism classes of stable sheaves on $\P_2$ with Hilbert polynomial $dm+c$. As shown in~\cite[Proposition~3.6]{LePotier}, $M_{dm+c}$ is a smooth projective variety of dimension $d^2+1$ in this case.

As shown in~\cite[Th\'eor\`eme~5.1]{LePotier}
$M_{dm+c}\iso \P(S^d V^*)$ for $d=1$, and $d=2$.
For $d=3$, $M_{3m\pm 1}$ is isomorphic to the universal cubic plane curve
\[
\{(C, p)\in \P(S^3V^*)\times \P_2 \mid p\in C\}.
\]

By~\cite{MaicanDuality} one has the isomorphisms $M_{4m+b}\iso M_{4m-1}$ for $d=4$ and $\gcd(4, b)=1$. In~\cite{DrezetMaican4m} a description of the moduli space $M_{4m-1}$ is given in terms of two strata: an open stratum $M_0$ and its closed complement $M_1$ in codimension $2$. The open stratum is naturally described as an open subvariety if a projective bundle $\bb B\ra N$ associated to a vector bundle of rank $12$ over a smooth $6$-dimensional projective variety $N$.

\subsection*{The main result of the paper}
The aim of this paper is to give a global description of $M=M_{4m-1}$ and hence of all fine Simpson moduli spaces of $1$-dimensional sheaves supported on plane quartics. We describe how the strata from~\cite{DrezetMaican4m} are ``glued together''. 
We notice that every sheaf from $M$ can be given as the cokernel of a morphism
\[
2\O_{\P_2}(-3)\oplus 3\O_{\P_2}(-2)
\xra{}
\O_{\P_2}(-2)\oplus 3\O_{\P_2}(-1),
\]
which gives a simple way to deform the sheaves from $M_0$ to the ones from $M_1$.
We show that $M$ is a blow-down to $M_1$ of the exceptional divisor $D$ of the blow-up $\tilde{\bb B}\defeq\Bl_{\bb B\setminus M_0} \bb B$ and give a geometric ``visual'' interpretation of this statement. This provides a different and a slightly more straightforward  way (compared to~\cite{ChoiChung}) of computing the Poincar\'e polynomial of $M$.

We were kindly informed by professor Kiryong Chung that the statement of Theorem~\ref{tr: main result} coincides with the statement of~\cite[Theorem~3.1]{ChungGlobal}, which must of course be given a priority because it appeared earlier. Our methods are however significantly different.

\subsection*{Structure of the paper}
In Section~\ref{section: two strata} we review the description of the strata of $M_{4m-1}$ from~\cite{DrezetMaican4m} and give a description of the degenerations to the closed stratum.
In Section~\ref{section: B} we present a geometric description of the fibres of the bundle $\bb B\ra N$ and construct local charts around the closed subvariety $\bb B'\defeq \bb B\setminus M_0$.
This allows us to give a geometric  description of the blow up $\tilde{\bb B}=\Bl_{\bb B'}\bb B$ and to ``see'' the main result, Theorem~\ref{tr: main result}, in Section~\ref{section: main result} just by looking at the geometric data involved. The computation of the Poincar\'e polynomial of $M$ is given here as a direct corollary from~Theorem~\ref{tr: main result}.
In Section~\ref{section: proof} we rigorously prove Theorem~\ref{tr: main result} by constructing a family of $(4m-1)$-sheaves on $\tilde{\bb B}$.

\section{$M_{4m-1}$ as a union of two strata}\label{section: two strata}
As shown in~\cite{LePotier} $M=M_{4m-1}$ is a smooth projective variety of dimension $17$. By~\cite{DrezetMaican4m} $M$ is  a disjoint union of two strata $M_1$ and $M_0$  such that $M_1$ is a closed subvariety of $M$ of codimension  $2$ and $M_0$ is its open complement.

\subsection{Closed stratum.}
The closed stratum $M_1$ is a closed subvariety of $M$ of codimension  $2$ given by the condition $h^0(\cal E)\neq 0$.

The sheaves from $M_1$ possess a locally free resolution
\begin{equation}\label{eq: res1}
0\ra 2\O_{\P_2}(-3)\xra{\smat{z_1&q_1\\z_2&q_2}} \O_{\P_2}(-2)\oplus \O_{\P_2}\ra \cal E\ra 0,
\end{equation}
with  linear independent linear forms $z_1$ and $z_2$  on $\P_2$.  $M_1$ is a geometric quotient of the variety of injective matrices $\smat{z_1&q_1\\z_2&q_2}$ as above by the non-reductive group
\[
( \Aut(2\O_{\P_2}(-3))\times \Aut(\O_{\P_2}(-2)\oplus \O_{\P_2}) )/\bb C^*.
\]
The points of $M_1$ are the isomorphism classes of sheaves that are non-trivial extensions
\begin{equation}\label{eq: closed stratum}
0\ra \O_C\ra \cal E\ra \O_{ \{p\} }\ra 0,
\end{equation}
where $C$ is a plane quartic given by the determinant of $\smat{z_1&q_1\\z_2&q_2}$ from~\refeq{eq: res1} and $p\in C$ a point on it given as the common zero set of $z_1$ and $z_2$.

This describes $M_1$ as the universal plane quartic, the quotient map is given by
\[
\pmat{z_1&q_1\\z_2&q_2}\mto (C, p),\quad  C=Z(z_1q_2-z_2q_1),\quad p=Z(z_1, z_2).
\]
$M_1$ is smooth of dimension $15$.

Let $M_{11}$ be the closed subvariety of $M_1$ defined by the condition that $p$ is contained on a line $L$ contained in $C$. Equivalently, in terms of the matrices~\refeq{eq: res1} this conditions reads as vanishing of $q_1$ or $q_2$. The dimension of $M_{11}$ is $12$.

\begin{lemma}
The sheaves in $M_{11}$ are non-trivial extensions
\[
0\ra \O_L(-2)\ra \cal F\ra \O_{C'}\ra 0,
\]
where $C'$ is a cubic and $L$ is a line.
\end{lemma}
\begin{proof}
Consider  the isomorphism class of  $\cal F$ with resolution
\begin{equation}
0\ra 2\O_{\P_2}(-3)\xra{\smat{l&0\\w&h}} \O_{\P_2}(-2)\oplus \O_{\P_2}\ra \cal F\ra 0.
\end{equation}
This gives the commutative diagram with exact rows and columns.
\[
\begin{xy}
(-35,30)*+{0}="2-1";
(0,30)*+{0}="20";
(20,30)*+{0}="21";
(-50,15)*+{0}="1-2";
(-35,15)*+{\O_{\P_2}(-3)}="1-1";
(0,15)*+{\O_{\P_2}(-2)}="10";
(20,15)*+{\O_L(-2)}="11";
(35,15)*+{0}="12";
(-50,0)*+{0}="0-2";
(-35,0)*+{2\O_{\P_2}(-3)}="0-1";
(0,0)*+{\O_{\P_2}(-2)\oplus \O_{\P_2}}="00";
(20,0)*+{\cal F}="01";
(35,0)*+{0}="02";
(-50,-15)*+{0}="-1-2";
(-35,-15)*+{\O_{\P_2}(-3)}="-1-1";
(0,-15)*+{\O_{\P_2}}="-10";
(20,-15)*+{\O_{C'}}="-11";
(35,-15)*+{0}="-12";
(-35,-30)*+{0}="-2-1";
(0,-30)*+{0}="-20";
(20,-30)*+{0}="-21";
{\ar@{->}^-{l}"1-1";"10"};
{\ar@{->}^-{\smat{1&0}}"1-1";"0-1"};
{\ar@{->}^-{\smat{1&0}}"10";"00"};
{\ar@{->}^-{}"0-2";"0-1"};
{\ar@{->}^-{\smat{l&0\\w&h}}"0-1";"00"};
{\ar@{->}^-{}"00";"01"};
{\ar@{->}^-{}"01";"02"};
{\ar@{->}^-{\smat{0\\1}}"00";"-10"};
{\ar@{->}^-{h}"-1-1";"-10"};
{\ar@{->}^-{\smat{0\\1}}"0-1";"-1-1"};
{\ar@{->}^-{}"-10";"-11"};
{\ar@{->}^-{}"01";"-11"};
{\ar@{->}^-{}"10";"11"};
{\ar@{->}^-{}"11";"01"};
{\ar@{->}^-{}"2-1";"1-1"};
{\ar@{->}^-{}"20";"10"};
{\ar@{->}^-{}"21";"11"};
{\ar@{->}^-{}"-1-1";"-2-1"};
{\ar@{->}^-{}"-10";"-20"};
{\ar@{->}^-{}"-11";"-21"};
{\ar@{->}^-{}"1-2";"1-1"};
{\ar@{->}^-{}"-1-2";"-1-1"};
{\ar@{->}^-{}"11";"12"};
{\ar@{->}^-{}"-11";"-12"};
\end{xy}
\]
Therefore, $\cal F$ is an extension
\[
0\ra \O_L(-2)\ra \cal F\ra \O_{C'}\ra 0,
\]
which is nontrivial since $\cal F$ is stable. This proves the required statement.
\end{proof}
Let $M_{10}$ denote  the open complement of $M_{11}$ in $M_1$.

\subsection{Open stratum.}
The open stratum $M_0$ is the complement of $M_1$ given by the condition $h^0(\cal E)=0$, it
 consists of the isomorphism classes of the cokernels of the injective morphisms
\begin{equation}\label{eq: res0}
\O_{\P_2}(-3)\oplus 2\O_{\P_2}(-2)\xra{A} 3\O_{\P_2}(-1)
\end{equation}
such that the $(2\times 2)$-minors of the linear part
$\smat{z_0&z_1&z_2\\w_0&w_1&w_2}$ of
\(
A=\smat{q_0&q_1&q_2 \\z_0&z_1&z_2\\w_0&w_1&w_2}
\)
are linear independent.

\subsubsection{$M_0$ as a geometric quotient.}
$M_0$ is an open subvariety in the geometric quotient $\bb B$ of the variety $\bb W^s$ of stable matrices as in \refeq{eq: res0} (see~\cite[Proposition~7.7]{MaicanTwoSemiSt} for details) by the group
\[
\Aut(\O_{\P_2}(-3)\oplus 2\O_{\P_2}(-2)) \times\Aut(3\O_{\P_2}(-1)).
\]
Its complement in $\bb B$ is a closed subvariety $\bb B'$ corresponding to the matrices with zero determinant.

\subsubsection{Extensions.}
If the maximal minors of the linear part of $A$ have a linear common factor, say  $l$, then $\det(A)=l\cdot h$ and $\cal E_A$ is in this case a non-split extension
\begin{equation}\label{eq: extension}
0\ra \O_L(-2)\ra \cal E_A\ra \O_{C'}\ra 0,
\end{equation}
where $L=Z(l)$, $C'=Z(h)$.

The subvariety $M_{01}$ of such sheaves is closed in $M_0$ and locally closed in $M$. Its boundary coincides with $M_{11}$.

\subsubsection{Twisted ideals of $3$ points on a quartic.}
Let  $M_{00}$ denote the open complement of $M_{01}$ in $M_0$.
In this case the maximal minors of the linear part  of $A$ are coprime, and  the cokernel $\cal E_A$ of~\refeq{eq: res0} is a part of the exact sequence
\[
0\ra \cal E_A\ra \O_C(1)\ra\O_Z\ra 0,
\]
where $C$ is a planar quartic curve
given by the determinant of $A$ from~\refeq{eq: res0} and $Z$ is
the zero dimensional subscheme of length $3$ given
by the maximal minors of the linear submatrix of $A$.
Notice that in this case the subscheme $Z$ does not lie on a line.

\subsection{Degenerations to the closed stratum}
\begin{pr}
1) Every sheaf in $M_1$ is a degeneration of sheaves from $M_{00}$. This corresponds to a degeneration of $Z\subset C$, where $Z$ is a zero-dimensional scheme of length $3$ not lying on a line and $C$ is a quartic curve,  to a flag $Z\subset C$ with $Z$ contained in a line $L$ that is not included in $C$. The limit corresponds to the point in $M_1$ described by the point $(L\cap C)\setminus Z$ on the quartic curve $C$.
\par
\noindent 2) The sheaves from $M_1$ given by pairs $(C, p)$ such that $p$ belongs to a line $L$ contained in $C$ are degenerations of sheaves from $M_{01}$. This corresponds to degenerations of extensions~\refeq{eq: extension} without sections to extensions with sections.
\end{pr}
The proof follows from the considerations below.
\subsubsection{Degenerations along $M_{00}$}
Fix a curve $C\subset \P_2$ of degree $4$, $C=Z(f)$, $f\in \Gamma(\P_2, \O_{\P_2}(4))$.
Let $Z\subset C$ be a zero-dimensional scheme of length $3$ contained in a line $L=Z(l)$, $l\in \Gamma(\P_2, \O_{\P_2}(1))$. Let $\cal F=\cal I_{Z}(1)$ be the twisted ideal sheaf of $Z$ in $C$ so that  there is an exact sequence
\[
0\ra \cal F\ra \O_C(1)\ra\O_Z\ra 0.
\]

\begin{lemma}
In the notations as above, the twisted ideal sheaf $\cal F=\cal I_Z(1)$  is semistable if and only if $L$ is not contained in $C$.
\end{lemma}
\begin{proof}
First of all notice that if $L$ is contained in $C$, then the inclusion of subschemes $Z\subset L\subset C$ provides a destabilizing subsheaf $\cal I_L(1)\subset \cal I_Z(1)$.

Let us construct now a locally free resolution of $\cal F$. Let $g\in \Gamma(\P_2, \O_{\P_2}(3))$ such that $\O_Z$ is given by the resolution
\[
0\ra \O_{\P_2}(-4)\xra{\smat{l& g}}\O_{\P_2}(-3)\oplus \O_{\P_2}(-1)\xra{\smat{-g\\l}}\O_{\P_2}\ra \O_Z\ra 0.
\]
Since $Z=Z(l, g)$ is contained in $C=Z(f)$, one concludes that $f=lh-wg$ for some $w\in \Gamma(\P_2, \O_{\P_2}(1))$ and $h\in \Gamma(\P_2, \O_{\P_2}(3))$.
This gives the following commutative diagram with exact rows and columns.
\[
\begin{xy}
(0,0)*+{\O_{\P_2}(1)}="00";
(20,0)*+{\O_C(1)}="01";
(-30,0)*+{\O_{\P_2}(-3)}="0-1";
(-45,0)*+{0}="0-2";
(-45,15)*+{0}="1-2";
(-45,30)*+{0}="2-2";
(35,0)*+{0}="02";
(35,15)*+{0}="12";
(0,-15)*+{\O_{Z}}="-10";
(20,-15)*+{\O_{Z}}="-11";
(0,-30)*+{0}="-20";
(20,-30)*+{0}="-21";
(0,15)*+{\O_{\P_2}(-2)\oplus \O_{\P_2}}="10";
(0,30)*+{\O_{\P_2}(-3)}="20";
(0,45)*+{0}="30";
(-30,45)*+{0}="3-1";
(-30,30)*+{\O_{\P_2}(-3)}="2-1";
(-30,15)*+{2\O_{\P_2}(-3)}="1-1";
(20,15)*+{\cal F}="11";
(20,30)*+{0}="21";
(-30,-15)*+{0}="-1-1";
(35,-15)*+{0}="-12";
{\ar@{->}^-{}"3-1";"2-1"};
{\ar@{->}^-{}"-1-1";"-10"};
{\ar@{->}^-{}"0-1";"-1-1"};
{\ar@{->}^-{}"21";"11"};
{\ar@{->}^-{}"-11";"-12"};
{\ar@{->}^-{}"00";"01"};
{\ar@{->}^-{lh-wg}"0-1";"00"};
{\ar@{->}^-{}"0-2";"0-1"};
{\ar@{->}^-{}"01";"02"};
{\ar@{->}^-{}"00";"-10"};
{\ar@{->}^-{}"01";"-11"};
{\ar@{=}^-{}"-10";"-11"};
{\ar@{->}^-{}"-10";"-20"};
{\ar@{->}^-{}"-11";"-21"};
{\ar@{->}^-{\smat{-g\\l}}"10";"00"};
{\ar@{->}^-{\smat{l&g}}"20";"10"};
{\ar@{->}^-{}"30";"20"};
{\ar@{=}^-{}"2-1";"20"};
{\ar@{->}^-{\smat{1&0}}"2-1";"1-1"};
{\ar@{->}^-{\smat{0\\1}}"1-1";"0-1"};
{\ar@{->}^-{\smat{l&g\\w&h}}"1-1";"10"};
{\ar@{->}^-{}"10";"11"};
{\ar@{->}^-{}"11";"01"};
{\ar@{->}^-{}"2-2";"2-1"};
{\ar@{->}^-{}"20";"21"};
{\ar@{->}^-{}"1-2";"1-1"};
{\ar@{->}^-{}"11";"12"};
\end{xy}
\]
Therefore, $\cal F$ possesses a locally free resolution
\[
0\ra 2\O_{\P_2}(-3)\xra{\smat{l&g\\w&h}} \O_{\P_2}(-2)\oplus \O_{\P_2}\ra \cal F\ra 0.
\]
In particular, if $l$ and $w$ are linear independent, which is true if and only if $f$ is not divisible by $l$, this is a resolution of type~\refeq{eq: res1}, hence $\cal F$ is a sheaf from $M_1$.

If $l$ and $w$ are linear dependent, then without loss of generality we can assume that $w=0$, which gives an extension
\[
0\ra \O_{C'}\ra \cal F\ra \O_L(-2)\ra 0, \quad C'=Z(h),
\]
and thus a destabilizing subsheaf $\O_{C'}$ of $\cal F$.
This concludes the proof.
\end{proof}

Let $H(3, 4)$ be the flag Hilbert scheme of zero-dimensional schemes of length $3$ on plane projective curves $C\subset \P_2$ of degree $4$.
Let $H'\subset H(3, 4)$ be the subscheme of those flags $Z\subset C$ such that $Z$ lies on a linear component of $C$. Using the universal family on $H(3, 4)$, one obtains a natural morphism
\[
H(3, 4)\setminus H'\ra M,
\]
whose image coincides with $M\setminus M_{01}$.

Its restriction to
the open subvariety $H_0(3, 4)$ of $H(3, 4)$ of flags
$Z\subset C\subset \P_2$ such that $Z$ does not lie on a line
gives an isomorphism
\[
H_0(3, 4)\ra M_{00}.
\]
Over $M_1$ one gets one-dimensional fibres: over an isomorphism class  in $M_1$, which is uniquely defined by a point $p\in C$ on a curve of degree $4$, the fibre can be identified with the variety of lines through $p$ that are not contained in $C$, i.~e., with a projective line without up to $4$ points.

\begin{rem}
Notice that the subvariety $H'$ is a $\P_3$-bundle over $\P V^*\times \P S^3V^*$, the fibre over the pair $(L, C')$ of a line $L$ and a cubic curve $C'$ is the Hilbert scheme $L^{[3]}$.

As shown in~\cite[Theorem~3.3  and Proposition~4.4]{ChoiChung}, the blow-up of $H(3,4)$ along $H'$ can be blown down along the fibres $L^{[3]}$  to the blow-up $\tilde M\defeq \Bl_{M_1} M$.
\end{rem}

\subsubsection{Degenerations along $M_{01}$}

For a fixed line $L$ and a fixed cubic curve $C'$ one can compute
$\Ext^1(\O_{C'}, \O_{L}(-2))\iso \k^3$. Therefore, using~\cite{LangeExt} one gets a projective bundle $P$ over $\P V^*\times \P S^3 V^*$ with fibre $\P_2$ and a universal family of extensions on it parameterizing the extensions
\[
0\ra \O_L(-2)\ra \cal F\ra \O_{C'}\ra 0, \quad L\in \P V^*, C'\in \P S^3V^*.
\]
This provides a morphism $P\ra M$ and describes the degenerations of sheaves from $M_{01}$ to sheaves in $M_{11}$.

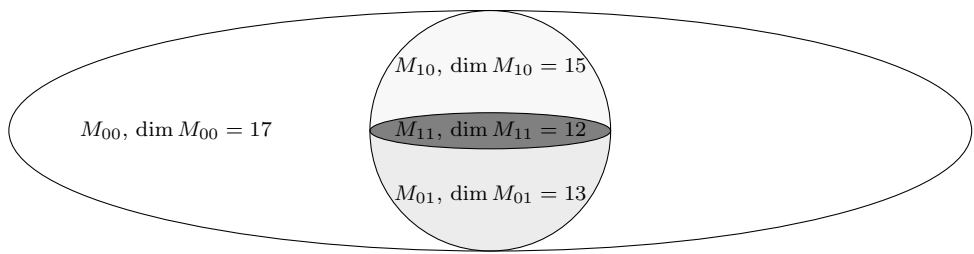
\begin{figure}[H]
\begin{center}
\begin{tikzpicture}[scale=0.8]
\def\firstellipse{(0,0) ellipse (8 and 2)}
\def\secondellipse{(0,0) ellipse (2 and 2)}
\def\thirdellipse{(0,0) ellipse (2 and 0.3)}
\begin{scope}
\clip (-2,0) rectangle (3,-5);
\fill[darkgray, fill opacity=0.1]\secondellipse;
\end{scope}
\begin{scope}
\clip (-2,0) rectangle (3,5);
\fill[lightgray, fill opacity=0.1]\secondellipse;
\end{scope}
\fill[gray] \thirdellipse;
\node[anchor=west] at (-7, 0) {\tiny $M_{00}$, $\dim M_{00}=17$};
\node[anchor=north] at (0, 1.4) {\tiny $M_{10}$, $\dim M_{10}=15$};
\node[anchor=south] at (0, -1.4) {\tiny $M_{01}$, $\dim M_{01}=13$};
\node at (0, 0) {\tiny $M_{11}$, $\dim M_{11}=12$};
\path[draw] \firstellipse;
\path[draw] \secondellipse;
\path[draw] \thirdellipse;
\end{tikzpicture}
\end{center}
\caption{Moduli space $M=M_{4m-1}(\P_2)$.}
\end{figure}

\section{Description of $\bb B$.}\label{section: B}
$\bb B$ is a projective bundle associated to a vector bundle of rank $12$ over the moduli space $N=N(3; 2, 3)$ of stable $(2\times 3)$ Kronecker modules, i.~e., over the GIT-quotient of the space $\bb V^s$ of stable $(2\times 3)$-matrices of linear forms on $\P_2$ by $\Aut(2\O_{\P_2}(-2)) \times\Aut(3\O_{\P_2}(-1))$.

The projection $\bb B\ra N$ is induced by
\[
\smat{q_0&q_1&q_2 \\z_0&z_1&z_2\\w_0&w_1&w_2}\mto \smat{z_0&z_1&z_2\\w_0&w_1&w_2}.
\]
For more details see~\cite[Proposition~7.7]{MaicanTwoSemiSt}.

\subsection{The base $N$.}

The subvariety $N'\subset N$ corresponding to the matrices whose minors have a common linear factor is isomorphic to $\P_2^*=\P V^*$, the space of lines in $\P_2$,  such that a line corresponds to the common linear factor of the minors of the corresponding Kronecker module $\smat{z_0&z_1&z_2\\w_0&w_1&w_2}$.

The blow up of $N$ along $N'$ is isomorphic to the Hilbert scheme $H=\P_2^{[3]}$ of $3$ points in $\P_2$ (cf.~\cite{EllStro}). The exceptional divisor $H'\subset H$ is a $\P_3$-bundle over $N'$, whose fibre over $\point{l}\in \P_2^*$ is the Hilbert scheme $L^{[3]}$ of $3$ points on $L=Z(l)$. The class in $N$ of a Kronecker module $\smat{z_0&z_1&z_2\\w_0&w_1&w_2}$ with coprime minors corresponds to the subscheme of $3$ non-collinear points in $\P_2$ defined by the minors of the matrix.

\subsection{The fibres of $\bb B\ra N$.} A fibre over a point from  $N\setminus N'$ can be seen as the space of plane quartics through the corresponding subscheme of $3$ non-collinear points. Indeed, consider a point from $N\setminus N'$ given by a Kronecker module $\smat{z_0&z_1&z_2\\w_0&w_1&w_2}$ with coprime minors $d_0, d_1, d_2$. The fibre over such a point consists of the orbits of injective matrices
\[
\pmat{q_0&q_1&q_2 \\z_0&z_1&z_2\\w_0&w_1&w_2}, \quad q_0, q_1, q_2\in S^2V^*,
\]
under the group action of
\[
\Aut(\O_{\P_2}(-3)\oplus 2\O_{\P_2}(-2)) \times\Aut(3\O_{\P_2}(-1)).
\]
If two matrices
\[
\smat{q_0&q_1&q_2 \\z_0&z_1&z_2\\w_0&w_1&w_2}, \quad \smat{Q_0&Q_1&Q_2 \\z_0&z_1&z_2\\w_0&w_1&w_2}
\]
lie in the same orbit of the group action, then their determinants are equal up to a multiplication by a non-zero constant. Vice versa, if the determinants of two such matrices are equal, $q-Q=(q_0-Q_0, q_1-Q_1, q_2-Q_2)$ lies in the syzygy module of  $\smat{d_0\\d_1\\ d_2}$, which is generated by the rows of $\smat{z_0&z_1&z_2\\w_0&w_1&w_2}$ by Hilbert-Burch theorem. This implies that $q-Q$ is a combination of the rows and thus the matrices lie on the same orbit.

A fibre over $\point{l}\in N'$ can be seen as the join $J(L^*, \P S^3 V^*)\iso \P_{11}$ of $L^*\iso \P H^0(L, \O_L(1))\iso \P_1$ and the space of plane cubic curves $\P(S^3 V^*)\iso \P_9$. To see this assume $l=x_0$, i.~e., $\point{x_0}$ is considered as the class of
\[
\pmat{-x_2&0&x_0\\x_1&-x_0&0}.
\]
Then the fibre over $[\smat{-x_2&0&x_0\\x_1&-x_0&0}]$ is given by the orbits of matrices
\begin{equation}\label{eq: fibre over N'}
\pmat{q_0(x_0, x_1, x_2)&q_1(x_1, x_2)&q_2(x_1, x_2)\\-x_2&0&x_0\\x_1&-x_0&0}
\end{equation}
and can be identified with the projective space
$\P(2H^0(L, \O_{L}(2))\oplus S^2V^*)$.
Rewrite the matrix~\refeq{eq: fibre over N'} as
\[
\pmat{q_0(x_0, x_1, x_2)&q_1(x_1,x_2) -x_2\cdot w&q_2(x_2)+x_1\cdot w\\-x_2&0&x_0\\x_1&-x_0&0}, \quad w(x_1, x_2)=\gamma x_1+\delta x_2, \quad \gamma, \delta\in \k.
\]
Its determinant equals
\[
x_0\cdot(x_0\cdot q_0(x_0, x_1, x_2)+x_1\cdot q_1(x_1, x_2)+ x_2\cdot q_2(x_2)).
\]
This allows to reinterpret the fibre as the projective space
\[
\P( H^0(L, \O_{L}(1))\oplus S^3V^* )\iso J(L^*, \P S^3V^*).
\]
$J(L^*, \P(S^3 V^*))\setminus L^*$ is a rank $2$ vector bundle over $\P(S^3 V^*)$, whose fibre over a cubic curve $C'\in \P S^3 V^*$
 is identified with the isomorphism classes of sheaves defined by~\refeq{eq: extension} from $M_{01}$ with fixed $L$ and $C'$, it
corresponds to the projective plane joining $C'$ with $L^*$ inside the join $J(L^*, \P(S^3 V^*))$. In the notations of the example above, $\gamma$ and $\delta$ are the coordinates of this plane.
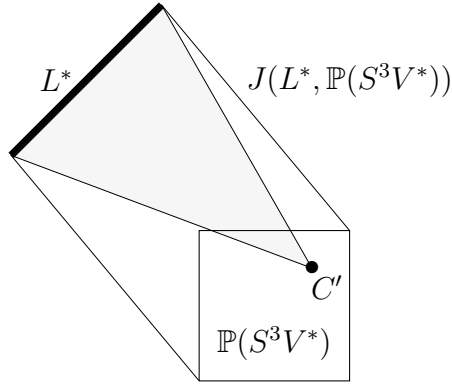
\begin{figure}[H]
\begin{center}
\begin{tikzpicture}
\draw[line width=1mm] ++(0, 0) to ++(2, 2);
\path node at (0.6,1) {$L^*$};
\path node at (4.5,1) {$J(L^*, \P(S^3V^*))$};
\draw[very thin] ++ (0, 0) to ++(2.5cm, -3cm);
\draw[very thin] ++(2, 2) to (4.5cm, -1cm);
\draw[very thin, fill=gray, fill opacity=0.07] (4, -1.5) to (0, 0) to (2, 2) to (4, -1.5);
\path node at (4, -1.5) {\textbullet};
\path node at (4.2, -1.75) {$C'$};
\begin{scope}[shift={(2.5,-3)}]
\draw ++(0, 0) to ++(2, 0) to ++(0, 2) to ++(-2, 0) to ++(0,-2);
\path node at (1,0.5) {$\P(S^3V^*)$};
\end{scope}
\end{tikzpicture}
\end{center}
\caption{The fibre of $\bb B$ over $L\in N'$.}
\end{figure}
The points of $J(L^*, \P(S^3 V^*))\setminus L^*$ parameterize the extensions~\refeq{eq: extension} from $M_{01}$ with fixed $L$.

\subsection{Description of $\bb B'$.}
$\bb B'$ is a union of lines $L^*$ from each fibre over $N'$ (as explained above).  It is isomorphic to the tautological $\P_1$-bundle over $N'=\P_2^*$
\begin{equation}\label{eq: B'}
\{(L, p)\in \P_2^*\times \P_2 \mid L\in \P_2^*, p\in L\}.
\end{equation}
Equivalently (cf.~\cite[p.~36]{DrezetMaican4m}), $\bb B'$ is isomorphic to the projective bundle associated to the tangent bundle $\Tsp \P_2^*$.
The fibre $\P_1$ of $\bb B'$ over, say, line $L=Z(x_0)\subset \P_2$ can be identified with the space of classes of matrices~\refeq{eq: res0} with zero determinant
\begin{equation}\label{eq: zero det}
\pmat{0&-x_2\cdot w&x_1\cdot w\\-x_2&0&x_0\\x_1&-x_0&0},\quad  w=\gamma x_1+\delta x_2,\quad \point{\gamma, \delta}\in \P_1.
\end{equation}

\subsection{Local charts around $\bb B'$.}

\begin{lemma}
Let $L\in N'$ be the class of the Kronecker module
\[
\pmat{-x_2&0&x_0\\x_1&-x_0&0}.
\]
Then there is an open neighbourhood of $L$ that can be identified with an open neighbourhood $U$
of zero in the affine space $\k^6$ via
\[
\k^6\supset U\ra N,\quad (\alpha,\beta, a,b,c,d)\mto \left[\pmat{-x_2&cx_1&\bar x_0\\x_1&-\bar x_0 + ax_1+bx_2 &dx_2}\right],\quad \bar x_0=x_0+\alpha x_1+\beta x_2,
\]
which establishes a local section of the quotient $\bb V^s\ra N$.
\end{lemma}
\begin{proof}
In some open neighbourhood $U$ of zero in $\k^6$ the morphism
\[
U\ra \bb V^s, \quad    (\alpha,\beta, a,b,c,d)\mto \pmat{-x_2&cx_1&\bar x_0\\x_1&-\bar x_0 + ax_1+bx_2 &dx_2}
\]
is well-defined.
Notice that two Kronecker modules of the form
\[
\pmat{-x_2&cx_1&\bar x_0\\x_1&-\bar x_0 + ax_1+bx_2 &dx_2}
\] can lie in the same orbit of the group action if and only if the matrices are equal.
Therefore, the morphism
\[
   (\alpha,\beta, a,b,c,d)\mto \left[\pmat{-x_2&cx_1&\bar x_0\\x_1&-\bar x_0 + ax_1+bx_2 &dx_2}\right]
\]
is injective.
\end{proof}
\begin{rem}
By abuse of notation we identify $U$ with its image in $N$.
\end{rem}
\begin{lemma}
$N'$ is cut out in $U$ by the equations $a=b=c=d=0$.
\end{lemma}
\begin{proof}
The maximal minors of $\smat{-x_2&cx_1&\bar x_0\\x_1&-\bar x_0 + ax_1+bx_2 &dx_2}$ are
\[
cdx_1x_2+\bar x_0(\bar x_0-ax_1-bx_2),\quad -dx_2^2-\bar x_0x_1,\quad x_2(\bar x_0-ax_1-bx_2)-cx_1^2.
\]
Clearly these minors have a common linear factor if $a$, $b$, $c$, $d$ vanish. On the other hand the condition $c=d=0$ is necessary to ensure the reducibility of these quadratic forms. If $c=d=0$, the conditions $a=b=0$ are necessary for the minors to have a common factor.
\end{proof}

\begin{lemma}
The restriction of $\bb B$ to $U$ is a trivial $\P_{11}$-bundle. Identifying $\P_{11}$ with the projective space
\[
\P(S^2V^* \oplus 2 \Span(x_1^2, x_1x_2, x_2^2)),
\]
i.~e., a point in $\P_{11}$ is identified with the class of the triple of quadratic forms
\[
(q_0(x_0, x_1, x_2), q_1(x_1, x_2), q_2(x_1, x_2)),
\]
one can identify $U\times \P_{11}$, and hence $\bb B|_{U}$, with the classes of matrices
\begin{equation}\label{eq: standard A}
\pmat{q_0(x_0, x_1, x_2)&q_1(x_1, x_2)&q_2(x_1, x_2)\\-x_2&cx_1&\bar x_0\\x_1&-\bar x_0 + ax_1+bx_2 &dx_2}.
\end{equation}
Assuming one of the coefficients of $q_0$, $q_1$, $q_2$ equal to $1$, we get local charts of the form $U\times \k^{11}$ and local sections of the quotient $\bb W^s\ra\bb B$.
\end{lemma}
\begin{proof}
It is enough to notice that as in~\refeq{eq: fibre over N'} one can get rid of $x_0$ in the expressions of $q_1$ and $q_2$.
\end{proof}

In order to get charts around $[A]\in \bb B'$,
\[
A=\pmat{0&-x_2\cdot w&x_1\cdot w\\-x_2&0&\bar x_0\\x_1&-\bar x_0&0},\quad  w=\gamma x_1+\delta x_2,\quad \point{\gamma, \delta}\in \P_1,
\]
rewrite~\refeq{eq: standard A} as
\begin{equation}\label{eq: standard A around B'}
\pmat{
q_0(x_0, x_1, x_2)&q_1(x_1,x_2) -x_2\cdot w&q_2(x_2)+x_1\cdot w\\
-x_2&cx_1&\bar x_0\\
x_1&-\bar x_0 + ax_1+bx_2 &dx_2}.
\end{equation}
Putting $\gamma=1$ or $\delta=1$, we get charts around $\bb B'$, each isomorphic to $U\times \k\times\k^{10}$. Denote them by $\bb B(\gamma)$ and $\bb B(\delta)$ respectively. Their coordinates  are those of $U$ together with $\delta$ respectively $\gamma$ and the coefficients of $q_i$, $i=0,1,2$.

The equations of $\bb B'$ are those of $N'$ in $U$, i.~e., $a=b=c=d=0$,  and the conditions imposed by vanishing of $q_0$, $q_1$, $q_2$.
\begin{rem}
Notice that these equations generate the ideal given by the vanishing of the determinant of~\refeq{eq: standard A around B'}.
\end{rem}

\section{Main result}\label{section: main result}
Consider the blow-up $\tilde{\bb B}=\Bl_{\bb B'}\bb B$. Let $D$ denote its exceptional divisor.

\begin{tr}\label{tr: main result}
$\tilde{\bb B}$ is isomorphic to the blow-up $\tilde M\defeq \Bl_{M_1}M$. The exceptional divisor of $\tilde M$ corresponds to $D$ under this isomorphism. The fibres of the morphism $D\ra M_1$ over the point of $M_1$ represented by a point $p$ on a quartic curve $C$ is identified with the projective line of lines in $\P_2$ passing through $p$.
\end{tr}
\subsection{A rather intuitive explanation}
Before rigorously proving this, let us explain how to arrive to  Theorem~\ref{tr: main result} and ``see'' it just by looking at the geometric data involved.  What follows in not quite rigorous but provides, in our opinion, a nice geometric picture.

Blowing up $\bb B$ along $\bb B'$ substitutes $\bb B'$ by the projective  normal
 bundle of $\bb B'$. So a point of $\bb B'$ represented by a line $L\in \P_2^*$ and a point $p\in L$, which is encoded by some $\point{w}\in \P H^*(L, \O_L(1))$,  is substituted by the projective space $D_{(L, p)}$ of the normal space $\Tsp_{(L, p)} \bb B/\Tsp_{(L, p)}\bb B'$ to $\bb B'$ at $(L, p)$.

 As $\bb B$ is a projective bundle over $N$, and $\bb B'$ is a $\P_1$-bundle over $N'$, the normal space is a direct sum of the normal spaces along the base and along the fibre. Therefore, $D_{(L, p)}$ is the join of the corresponding projective spaces: of $\P_3=L^{[3]}$ (normal projective space to $N'$ in $N$ at $L\in N'$) and $\P_9=\P(S^3 V^*)$ (normal projective space to $L^*$ in $J(L^*, \P(S^3 V^*))$ at $p\in L\iso L^*$; notice that the normal bundle of  $L^*\subset J(L^*, \P(S^3 V^*))$, i.~e., $\P_1\subset \P_{11}$, is trivial).

The space $L^{[3]}$  is naturally identified with the projective space of cubic forms on $L^*$ whereas
$\P(S^3 V^*)$ is clearly the space of cubic curves on $\P_2$.

Assume  $L=Z(x_0)$ such that $\{x_0, x_1, x_2\}$ is a basis of $V^*=H^0(\P_2, \O_{\P_2}(1))$. Identifying $x_1$ and $x_2$ with their images in $H^0(L, \O_L(1))$, $\{x_1, x_2\}$ is a basis of $H^0(L, \O_L(1))$.

We conclude that the join of $L^{[3]}$ and $\P S^3 V^*$  can  be  identified with the projective space corresponding to the vector space
\[
\{\lambda\cdot x_0 h+ \lambda'\cdot w g \mid h(x_0, x_1, x_2)\in S^3 V^*,  g(x_1, x_2)\in H^0(L, \O_L(3)),  \point{\lambda, \lambda'}\in \P_1\},
\]
i.~e., the space of planar quartic curves through the point $p=Z(x_0, w)$.

So the exceptional divisor of the blow-up $\Bl_{\bb B'}\bb B$ is a projective bundle with fibre over $(L, p)$ being interpreted as the space of quartic curves through $p$.

The fibre of  $\bb B\ra N$ over $L\in N'$ is substituted by the fibre that consists of two components: the first component is the blow-up of $J(L^*, \P(S^3 V^*))$ along $L^*$,  the second one is a projective bundle over $L^*$ with the fibre $\P_{13}=J( L^{[3]}, \P(S^3 V^*) )$, the components intersect along  $L^*\times \P(S^3 V^*)$.
\begin{figure}[H]
\begin{center}
\begin{tikzpicture}[scale=1.2]
\begin{scope}
\draw ++(0, 0) to ++(2, 0) to ++(0, 2) to ++(-2, 0) to ++(0,-2);
\foreach \x in {1,..., 3}
{
\draw ++(2cm+3*\x mm ,2cm +3*\x mm ) to ++(-2, 0);
\draw ++(2cm+3*\x mm ,2cm +3*\x mm ) to ++(0, -2);
\draw ++(3*\x mm ,2cm +3*\x mm ) to ++(0, -3 mm);
\draw ++(2cm+3*\x mm ,3*\x mm) to ++(-3mm, 0);
}
\path node at (1,0.6) {$\P(S^3 V^*)$};
\end{scope}
\begin{scope}[shift={(5,0)}]
\draw ++(0, 2) to ++(2, 0) to ++(-0.5, -2) to (0, 2);
\foreach \x in {1,..., 3}
{
 \draw ++(-3*\x mm, 2cm +3*\x mm) to ++(2, 0) to ++(-0.75mm, -3mm);
 \draw ++(-3*\x mm, 2cm +3*\x mm) to ++(1.5, -2);
 \draw ++(1.5 cm -3*\x mm, 3*\x mm) to ++( 0.1865 mm, 0.75 mm);
}
\path node at (1.3,1.2) {$L^{[3]}$};
\end{scope}
\foreach \x in {0,..., 3}
{
\draw[very thin] ++(2cm + 3*\x mm,2cm +3*\x mm) to ++(3cm, 0);
\draw[very thin] ++(2cm + 3*\x mm,3*\x mm) to ++(4.5cm -6*\x mm, 0);
}
\begin{scope}[shift={(2.5,-3)}]
\draw ++(0, 0) to ++(2, 0) to ++(0, 2) to ++(-2, 0) to ++(0,-2);
\path node at (1,1) {$\P(S^3 V^*)$};
\end{scope}
\draw[line width=1mm] (-0.3, 1.7) to ++(16mm, 16mm);
\draw[very thin] ++ (0, 0) to ++(2.5cm, -3cm);
\foreach \x in {0,..., 3}
{
\draw[very thin] ++(4.5, -1) to (2cm+3*\x mm, 2cm+3*\x mm);
}
\end{tikzpicture}
\caption{The fibre of $\tilde{\bb B}$ over $L\in N'$.}
\end{center}
\end{figure}
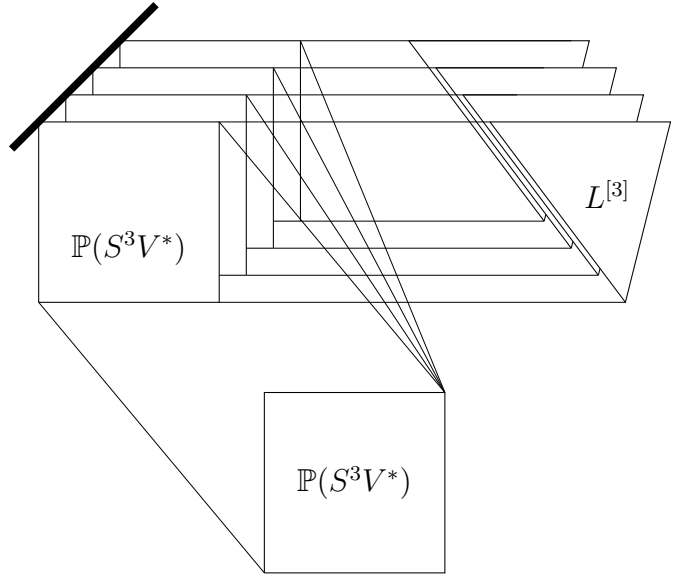

Consider the map $\bb B'\ra \P_2$ given by the projection to the second factor (cf.~\refeq{eq: B'}). The fibre $\bb B'_p$ over a point $p\in \P_2$ is a projective line
\(
\{(L, p) \mid p\in L\}
\)
in $\bb B'$, every two points $(L, p)$ and $(L', p)$ of $\bb B'_p\subset \bb B'$  are substituted by the projective spaces $J( L^{[3]}, \P(S^3 V^*) )$ and $J( L'^{[3]}, \P(S^3 V^*) )$ respectively, each of which is naturally identified with
the space of quartics through $p$.
Assume without loss of generality $p=\point{0,0,1}$.

The fibre $\bb B'_p$ in this case is identified with the space of  lines in $\P_2$ through $p$, i.~e., with the projective line in $N'=\P_2^*=\P V^*$ consisting of classes of linear forms $\alpha x_0+\beta x_1$, $\point{\alpha, \beta}\in \P_1$. The fibre has a standard covering $\bb B'_{p,0}=\{x_0+\beta x_1\}$, $\bb B'_{p,1}=\{\alpha x_0+x_1\}$, which is induced by the standard covering of $\P_2^*$. The elements of the fibre  corresponding to the points of $\bb B'_{p, 0}$ are the equivalence classes of matrices
\[
A_0=\pmat{
0& x_1\cdot x_2&-x_1\cdot x_1\\
-x_2&0&x_0+\beta x_1\\
x_1&-(x_0+\beta x_1)&0}.
\]
The elements of the fibre  corresponding to the points of $\bb B'_{p,1}$ are the equivalence classes of matrices
\[
A_1=\pmat{
0& x_0\cdot x_2&-x_0\cdot x_0\\
-x_2&0&\alpha x_0+x_1\\
x_0&-(\alpha x_0+x_1)&0}.
\]
This way we have chosen so to say the normal forms for the representatives of the points in the fibre $\bb B'_p$.
For $\beta=\alpha^{-1}$, i.~e. on the intersection of $\bb B'_{p,0}$ and $\bb B'_{p,0}$, these matrices are equivalent. One computes that
\(
g A_1 h=A_0
\)
for  matrices
\[
g=\pmat{\alpha^3&\alpha^2x_0-\alpha x_1&\alpha^2x_2\\
0&1&0\\
0&0&-\alpha},
\quad
h=\pmat{1&0&0\\
\alpha^{-1}&-\alpha^{-2}&0\\
0&0&\alpha^{-1}
}
\]
with determinants
\[
\det g=-\alpha^4, \quad \det h=-\alpha^{-3}.
\]
Consider the automorphism $\bb W^s\xra{\xi} \bb W^s$, $A\mto g A h$.
Then
\begin{equation}\label{eq: cocycle on fibres}
\det(\xi(A_1+B_1))=\alpha\cdot\det(A_1+B_1).
\end{equation}

From \refeq{eq: cocycle on fibres} it follows that
the restriction of the ideal sheaf of $D$ to a fibre of the morphisms $D\ra M_1$ is $\O_{\P_1}(1)$. By ~\cite{Nakano, NakanoSupl1, NakanoSupl2}, this means that one can blow down $D$ in $\tilde{\bb B}$ along the map $D\ra M_1$.
This gives the blow down $\Bl_{\bb B'}\bb B\ra M$ that
contracts the exceptional divisor of $\Bl_{\bb B'}\bb B$ along all lines $\bb B'_p$.

\subsection{Poincar\'e polynomial of $M$}
Theorem~\ref{tr: main result} provides an easy way to compute the Poincar\'e polynomial $P(M)$ of $M$. Though our computation is similar to the one from~\cite[Corollary~5.2]{ChoiChung}, it is a bit more straightforward. Notice that $P(M)$ has been also computed using a torus action on $M$ in~\cite[Theorem~1.1]{ChoiMaican}.

\begin{cor}
The Poincar\'e polynomial of $M$ equals
$1+2q+6q^2+10q^3+14q^4+15q^5+16q^7+16q^8+16q^9+16q^{10}+16q^{11}+15q^{12}+14q^{13}+10q^{14}+6q^{15}+2q^{16}+q^{17}$.
\end{cor}
\begin{proof}
By~\cite[page 90]{Elencwajg} the Poincar\'e polynomial of $H$ is
\(
P(H)=1+2q+5q^2+6q^3+5q^4+2q^5+q^6.
\)
Then $P(N)=P(H)-P(\P_2)\cdot P(\P_3)+P(\P_2)$. Since $\bb B$ is a projective bundle over $N$ with fibre $\P_{11}$, one gets $P(\bb B)=P(N)\cdot P(\P_{11})$. By blowing up $\bb B$ along $\bb B'$ and blowing down the result to $M$ one substitutes $\bb B'$ by $M_1$, which is a $\P_{13}$-bundle over $\P_2$. Therefore, $P(M)=P(\bb B)-P(\P_2)\cdot P(\P_1)+P(\P_{2})\cdot P(\P_{13})$. Using $P(\P_n)=\frac{1-q^n}{1-q}$ we get the result.
\end{proof}

\section{The proof}\label{section: proof}
Now let us properly prove Theorem~\ref{tr: main result}.

\subsection{Exceptional divisor $D$ and quartic curves}
Notice  that the subvariety $\bb W'$ in $\bb W^s$ parameterizing $\bb B'$ is given by the condition $\det A=0$.
This way we obtain a morphism $\tilde{\bb B}\ra \P S^4V^*$.
\begin{lemma}\label{lemma: divisor D}
1) The restriction of  $\tilde{\bb B}\ra \P S^4V^*$ to $D$ maps a point of $D$ lying over a point $p\in \P_2$ \textup{(}via the map $D\ra \bb B'\ra \P_2$\textup{)} to a quartic curve through $p$, i.~e., there is a morphism $D\ra M_1\subset \P_2\times \P S^4V^*$.

2) The fibre $D_{(L, p)}$ of $D\ra \bb B'$ over $(L, p)\in \bb P_2\times\P^*_2$, $p\in L$, is isomorphic via the map $\tilde{\bb B}\ra \P S^4V^*$ to the linear subspace in $\P S^d V^*$ of curves through $p$.

3) The morphism $D\ra M_1$ is a $\P_1$-bundle over $M_1$,
its fibre over a point of $M_1$ given by a pair $p\in C$ can be identified with the fibre of $\bb B'\ra \P_2$ over $p$.
\end{lemma}
\begin{proof}
1) Let $[A]\in \bb B'$ with $A$ as in~\refeq{eq: zero det} and let $a_0, a_1, a_2$ be the rows of $A$.
Let $B$ be a tangent vector at $A$, which can be identified with a morphism of type~\refeq{eq: res0}. Let $b_0, b_1, b_2$ be its rows. Then, since $\det A=0$,
\[
\det(A+tB)=
f_{A,B}\cdot t\mod (t^2),\quad\text{for }
f_{A,B}=
\det\smat{b_0\\a_1\\a_2}+
\det\smat{a_0\\b_1\\a_2}+
\det\smat{a_0\\a_1\\b_2}.
\]
Then $f_{A, B}$ is a non-zero quartic form if $B$ is normal to $\bb W'$. One computes
\[
f_{A, B}=
x_0\sum_{i=0}^2x_ib_{0i}-w(x_1\sum_{i=0}^2 x_ib_{1i}+x_2\sum_{i=0}^2x_ib_{2i})
\]
and thus $f_{A, B}$  vanishes at $p$, which is the common zero point of $x_0$ and $w$.

2)  Since the map $D\ra \P S^4V^*$ is injective, it is enough to notice that,
for a fixed $A\in \bb W'$, every quartic form through $p$ can be obtained by varying $B$. This gives a bijection and thus an isomorphism from $D_{(L,p)}$ to the space of quartics through $p$.

3) Follows from 1) and 2). 
\end{proof}

\subsection{Local charts.}\label{subsection: local charts of the blow-up}
Let us describe $\tilde{\bb B}$ over $\bb B(\delta)$. Around points of  $D$ lying over $[A]\in \bb B(\delta)$ there are $14$ charts.
For a fixed coordinate $t$ of $\bb B(\delta)$ different from $\alpha$, $\beta$, $\gamma$, denote the corresponding chart of  $\Bl_{\bb B'\cap B(\delta) }\bb B(\delta)$ by $\tilde{\bb B}(t)$.
Then $\tilde{\bb B}(t)$ can be identified with the variety of triples $(A, t, B)$,
\begin{align}\label{eq: blow up coordinates}
\begin{split}
A&=\pmat{
0&-x_2\cdot (\gamma x_1+x_2)&x_1\cdot (\gamma x_1+x_2)\\
-x_2&0&\bar x_0\\
x_1&-\bar x_0&0},
\\
B&=\pmat{
q_0(x_0, x_1, x_2)&q_1(x_1,x_2) &q_2(x_2)\\
0&cx_1&0\\
0&ax_1+bx_2 &dx_2},
\end{split}
\end{align}
such that the coefficient of  $B$ corresponding to $t$ equals $1$ and $A+t\cdot B$ belongs to $\bb B(\delta)$.
The blow-up map $\tilde{\bb B}(t)\ra \bb B(t)$ is given under this identification by  sending a triple $(A, t,B)$ to $A+t\cdot B$.

\subsection{Family of $(4m-1)$-sheaves on $\tilde{\bb B}$.}
Notice that the cokernel of~\refeq{eq: res0} is isomorphic to the cokernel of
\[
2\O_{\P_2}(-3)\oplus 3\O_{\P_2}(-2)
\xra{
\pmat{
0&0&0&0\\
0&q_0&q_1&q_2\\
0&z_0&z_1&z_2\\
0&w_0&w_1&w_2\\
1&0&0&0
}
}
\O_{\P_2}(-2)\oplus 3\O_{\P_2}(-1).
\]
\subsubsection{Local construction}
\begin{lemma}
For $t\neq 0$ consider the matrix
\begin{equation}\label{eq: deformation at B'}
\pmat{
0&0&0&0\\
0&0&-wx_2&wx_1\\
0&-x_2&0&\bar x_0\\
0&x_1&-\bar x_0&0\\
1&0&0&0
}+
t\cdot
\pmat{
0&0&0&0\\
0&q_0&q_1&q_2\\
0&y_0&y_1&y_2\\
0&z_0&z_1&z_2\\
0&0&0&0
}
\end{equation}
as a morphism
\(
2\O_{\P_2}(-3)\oplus 3\O_{\P_2}(-2)
\xra{}\O_{\P_2}(-2)\oplus 3\O_{\P_2}(-1).
\)
Then its cokernel is isomorphic to the cokernel of
\begin{equation}\label{eq: deformation at B' new}
\pmat{
\bar x_0&0&0&0\\
w&0&0&0\\
0&-x_2&0&\bar x_0\\
0&x_1&-\bar x_0&0\\
t&0&x_2&-x_1
}+
\pmat{
0&x_1y_0+x_2z_0&x_1y_1+x_2z_1&x_1y_2+x_2z_2\\
0&q_0&q_1&q_2\\
0&ty_0&ty_1&ty_2\\
0&tz_0&tz_1&tz_2\\
0&0&0&0
}.
\end{equation}
\end{lemma}
\begin{proof}
Acting by the automorphisms of $2\O_{\P_2}(-3)\oplus 3\O_{\P_2}(-2)$ on the left and by the automorphisms of $\O_{\P_2}(-2)\oplus 3\O_{\P_2}(-1)$ on the right of~\refeq{eq: deformation at B'}, we transform this matrix as follows:
\begin{align*}
&\smat{
0&0&0&0\\
0&0&-wx_2&wx_1\\
0&-x_2&0&\bar x_0\\
0&x_1&-\bar x_0&0\\
1&0&0&0
}+
t\cdot
\smat{
0&0&0&0\\
0&q_0&q_1&q_2\\
0&y_0&y_1&y_2\\
0&z_0&z_1&z_2\\
0&0&0&0
}
\sim
\smat{
0&0&0&0\\
w&0&-wx_2&wx_1\\
0&-x_2&0&\bar x_0\\
0&x_1&-\bar x_0&0\\
1&0&0&0
}+
t\cdot
\smat{
0&0&0&0\\
0&q_0&q_1&q_2\\
0&y_0&y_1&y_2\\
0&z_0&z_1&z_2\\
0&0&0&0
}\\
\sim
&\smat{
0&0&0&0\\
w&0&0&0\\
0&-x_2&0&\bar x_0\\
0&x_1&-\bar x_0&0\\
1&0&x_2&-x_1
}+
t\cdot
\smat{
0&0&0&0\\
0&q_0&q_1&q_2\\
0&y_0&y_1&y_2\\
0&z_0&z_1&z_2\\
0&0&0&0
}
\sim
\smat{
\bar x_0&0&\bar x_0x_2&-\bar x_0x_1\\
w&0&0&0\\
0&-x_2&0&\bar x_0\\
0&x_1&-\bar x_0&0\\
1&0&x_2&-x_1
}+
t\cdot
\smat{
0&0&0&0\\
0&q_0&q_1&q_2\\
0&y_0&y_1&y_2\\
0&z_0&z_1&z_2\\
0&0&0&0
}\\
\sim
&\smat{
\bar x_0&0&0&0\\
w&0&0&0\\
0&-x_2&0&\bar x_0\\
0&x_1&-\bar x_0&0\\
1&0&x_2&-x_1
}+
t\cdot
\smat{
0&x_1y_0+x_2z_0&x_1y_1+x_2z_1&x_1y_2+x_2z_2\\
0&q_0&q_1&q_2\\
0&y_0&y_1&y_2\\
0&z_0&z_1&z_2\\
0&0&0&0
}\\
\sim
&\smat{
t^{-1}\bar x_0&0&0&0\\
t^{-1}w&0&0&0\\
0&-x_2&0&\bar x_0\\
0&x_1&-\bar x_0&0\\
1&0&x_2&-x_1
}+
\smat{
0&x_1y_0+x_2z_0&x_1y_1+x_2z_1&x_1y_2+x_2z_2\\
0&q_0&q_1&q_2\\
0&ty_0&ty_1&ty_2\\
0&tz_0&tz_1&tz_2\\
0&0&0&0
}\\
\sim
&\smat{
\bar x_0&0&0&0\\
w&0&0&0\\
0&-x_2&0&\bar x_0\\
0&x_1&-\bar x_0&0\\
t&0&x_2&-x_1
}+
\smat{
0&x_1y_0+x_2z_0&x_1y_1+x_2z_1&x_1y_2+x_2z_2\\
0&q_0&q_1&q_2\\
0&ty_0&ty_1&ty_2\\
0&tz_0&tz_1&tz_2\\
0&0&0&0
},
\end{align*}
which concludes the proof.
\end{proof}

Evaluating~\refeq{eq: deformation at B' new} at $t=0$ gives
\[
\pmat{
\bar x_0&x_1y_0+x_2z_0&x_1y_1+x_2z_1&x_1y_2+x_2z_2\\
w&q_0&q_1&q_2\\
0&-x_2&0&\bar x_0\\
0&x_1&-\bar x_0&0\\
0&0&x_2&-x_1
}.
\]
\begin{lemma}\label{lemma: res0new}
The isomorphism class of the  cokernel $\cal F$ of
\[
2\O_{\P_2}(-3)\oplus 3\O_{\P_2}(-2)
\xra{
\pmat{
\bar x_0&p_0&p_1&p_2\\
w&q_0&q_1&q_2\\
0&-x_2&0&\bar x_0\\
0&x_1&-\bar x_0&0\\
0&0&x_2&-x_1
}
}
\O_{\P_2}(-2)\oplus 3\O_{\P_2}(-1)
\]
is a sheaf from $M_1$ with resolution
\begin{equation}\label{eq: res1new}
0\ra 2\O_{\P_2}(-3)\xra{\smat{\bar x_0&g\\w&h}} \O_{\P_2}(-2)\oplus \O_{\P_2}\ra\cal F\ra 0,
\end{equation}
if $\bar x_0h-wg\neq 0$ for $g=x_0p_0+x_1p_1+x_2p_2$, $h=x_0q_0+x_1q_1+x_2q_2$.
\end{lemma}
\begin{proof}
Consider  the isomorphism class of  $\cal F$ with resolution~\refeq{eq: res1new}.
Then, using the Koszul resolution of $\O_{\P_2}$,
one concludes that the kernel of the composition of two surjective morphisms
\[
\O_{\P_2}(-2)\oplus 3\O_{\P_2}(-1) \xra{\smat{1&0\\0&x_0\\0&x_1\\0&x_2}}
\O_{\P_2}(-2)\oplus \O_{\P_2}\ra \cal F
\]
coincides with the image of
\[
2\O_{\P_2}(-3)\oplus 3\O_{\P_2}(-2)
\xra{
\smat{
\bar x_0&p_0&p_1&p_2\\
w&q_0&q_1&q_2\\
0&-x_2&0&\bar x_0\\
0&x_1&-\bar x_0&0\\
0&0&x_2&-x_1
}
}
\O_{\P_2}(-2)\oplus 3\O_{\P_2}(-1),
\]
which concludes the proof.
\end{proof}

For $A+tB$ with $A$ and $B$ as in~\refeq{eq: blow up coordinates} we obtain the morphism
\[
2\O_{\P_2}(-3)\oplus 3\O_{\P_2}(-2)
\xra{
\pmat{
\bar x_0&0&cx_1^2+ax_1x_2+bx_2^2&dx_2^2\\
w&q_0&q_1&q_2\\
0&-x_2&tcx_1&\bar x_0\\
0&x_1&-\bar x_0+t(ax_1+bx_2)&tdx_2\\
t&0&x_2&-x_1
}
}
\O_{\P_2}(-2)\oplus 3\O_{\P_2}(-1),
\]
which defines by Lemma~\ref{lemma: res0new} a family of $(4m-1)$-sheaves on $\tilde{\bb B}(t)$ and therefore a morphism $\tilde{\bb B}(t)\ra M$. This morphism sends the point of the exceptional divisor represented by $(A, 0, B)$ to the point given by the quartic curve $C=Z(f)$,
\[
f=\bar x_0\cdot (\bar x_0q_0(x_0, x_1, x_2)+x_1q_1(x_1, x_2)+x_2q_2(x_2))-w\cdot (cx_1^3+ax_1^2x_2+bx_1x_2^2+dx_2^3)
\]
and the point $p=Z(\bar x_0, w)$ on $C$.

\subsubsection{Gluing the morphisms $\tilde{\bb B}(t)\ra M$}
For different charts $\tilde{\bb B}(t)$ and $\tilde{\bb B}(t')$ the corresponding morphisms agree on intersections.  Therefore we conclude that there exists a morphism $\tilde{\bb B}\ra M$. It is an isomorphism outside of $D$. As already mentioned in Lemma~\ref{lemma: divisor D}, the restriction of this morphism to $D$ gives a $\P_1$-bundle $D\ra M_1$.
\begin{lemma}
The map $\tilde{\bb B}\ra M$ is the blow-up $\Bl_{M_1}M\ra M$.
\end{lemma}
\begin{proof}
By the universal property of blow-ups, there exists a unique morphism
$\tilde{\bb B}\xra{\phi} \Bl_{M_1}M$ over $M$, which maps $D$ to the exceptional divisor $E$  of $\Bl_{M_1}M$ and is an isomorphism outside of $D$. This morphism must be surjective as its image is irreducible and contains an open set. Therefore, the fibres of $\phi$ over $E$ must be zero-dimensional and connected by the Zariski main theorem. This means that $\phi$ is a bijective morphism of smooth varieties and hence an isomorphism.
\end{proof}
This concludes the proof of Theorem~\ref{tr: main result}.

\def\cprime{$'$} \def\cprime{$'$} \def\cprime{$'$}

\end{document}